\documentclass[12pt]{amsart}
\usepackage{amsmath,amssymb,amsbsy,amsfonts,latexsym,amsopn,amstext,cite,
                                               amsxtra,euscript,amscd,bm}
\usepackage{url}
\usepackage[colorlinks,linkcolor=blue,anchorcolor=blue,citecolor=blue,backref=page]{hyperref}
\usepackage{color}
\usepackage{graphics,epsfig}
\usepackage{graphicx}
\usepackage{float}

\hypersetup{breaklinks=true}

\renewcommand*{\backref}[1]{}
\renewcommand*{\backrefalt}[4]{%
    \ifcase #1 (Not cited.)%
    \or        (p.\,#2)%
    \else      (pp.\,#2)%
    \fi}

\newfont{\teneufm}{eufm10}
\newfont{\seveneufm}{eufm7}
\newfont{\fiveeufm}{eufm5}
\newfam\eufmfam
                \textfont\eufmfam=\teneufm \scriptfont\eufmfam=\seveneufm
                \scriptscriptfont\eufmfam=\fiveeufm
\def\bbbc{{\mathchoice {\setbox0=\hbox{$\displaystyle\rm C$}\hbox{\hbox
to0pt{\kern0.4\wd0\vrule height0.9\ht0\hss}\box0}}
{\setbox0=\hbox{$\textstyle\rm C$}\hbox{\hbox
to0pt{\kern0.4\wd0\vrule height0.9\ht0\hss}\box0}}
{\setbox0=\hbox{$\scriptstyle\rm C$}\hbox{\hbox
to0pt{\kern0.4\wd0\vrule height0.9\ht0\hss}\box0}}
{\setbox0=\hbox{$\scriptscriptstyle\rm C$}\hbox{\hbox
to0pt{\kern0.4\wd0\vrule height0.9\ht0\hss}\box0}}}}
\def\bbbq{{\mathchoice {\setbox0=\hbox{$\displaystyle\rm
Q$}\hbox{\raise 0.15\ht0\hbox to0pt{\kern0.4\wd0\vrule
height0.8\ht0\hss}\box0}} {\setbox0=\hbox{$\textstyle\rm
Q$}\hbox{\raise 0.15\ht0\hbox to0pt{\kern0.4\wd0\vrule
height0.8\ht0\hss}\box0}} {\setbox0=\hbox{$\scriptstyle\rm
Q$}\hbox{\raise 0.15\ht0\hbox to0pt{\kern0.4\wd0\vrule
height0.7\ht0\hss}\box0}} {\setbox0=\hbox{$\scriptscriptstyle\rm
Q$}\hbox{\raise 0.15\ht0\hbox to0pt{\kern0.4\wd0\vrule
height0.7\ht0\hss}\box0}}}}
\def\bbbt{{\mathchoice {\setbox0=\hbox{$\displaystyle\rm
T$}\hbox{\hbox to0pt{\kern0.3\wd0\vrule height0.9\ht0\hss}\box0}}
{\setbox0=\hbox{$\textstyle\rm T$}\hbox{\hbox
to0pt{\kern0.3\wd0\vrule height0.9\ht0\hss}\box0}}
{\setbox0=\hbox{$\scriptstyle\rm T$}\hbox{\hbox
to0pt{\kern0.3\wd0\vrule height0.9\ht0\hss}\box0}}
{\setbox0=\hbox{$\scriptscriptstyle\rm T$}\hbox{\hbox
to0pt{\kern0.3\wd0\vrule height0.9\ht0\hss}\box0}}}}
\def\bbbs{{\mathchoice
{\setbox0=\hbox{$\displaystyle     \rm S$}\hbox{\raise0.5\ht0\hbox
to0pt{\kern0.35\wd0\vrule height0.45\ht0\hss}\hbox
to0pt{\kern0.55\wd0\vrule height0.5\ht0\hss}\box0}}
{\setbox0=\hbox{$\textstyle        \rm S$}\hbox{\raise0.5\ht0\hbox
to0pt{\kern0.35\wd0\vrule height0.45\ht0\hss}\hbox
to0pt{\kern0.55\wd0\vrule height0.5\ht0\hss}\box0}}
{\setbox0=\hbox{$\scriptstyle      \rm S$}\hbox{\raise0.5\ht0\hbox
to0pt{\kern0.35\wd0\vrule height0.45\ht0\hss}\raise0.05\ht0\hbox
to0pt{\kern0.5\wd0\vrule height0.45\ht0\hss}\box0}}
{\setbox0=\hbox{$\scriptscriptstyle\rm S$}\hbox{\raise0.5\ht0\hbox
to0pt{\kern0.4\wd0\vrule height0.45\ht0\hss}\raise0.05\ht0\hbox
to0pt{\kern0.55\wd0\vrule height0.45\ht0\hss}\box0}}}}
\def\bbbz{{\mathchoice {\hbox{$\sf\textstyle Z\kern-0.4em Z$}}
{\hbox{$\sf\textstyle Z\kern-0.4em Z$}} {\hbox{$\sf\scriptstyle
Z\kern-0.3em Z$}} {\hbox{$\sf\scriptscriptstyle Z\kern-0.2em
Z$}}}}

\newtheorem{theorem}{Theorem}
\newtheorem{lemma}[theorem]{Lemma}

\newtheorem{cor}[theorem]{Corollary}

\newtheorem{rem}[theorem]{Remark}

\numberwithin{equation}{section}
\numberwithin{theorem}{section}
\numberwithin{table}{section}

\def\squareforqed{\hbox{\rlap{$\sqcap$}$\sqcup$}}
\def\qed{\ifmmode\squareforqed\else{\unskip\nobreak\hfil
\penalty50\hskip1em\null\nobreak\hfil\squareforqed
\parfillskip=0pt\finalhyphendemerits=0\endgraf}\fi}

\def\cA{{\mathcal A}}
\def\cB{{\mathcal B}}
\def\cC{{\mathcal C}}

\def\cG{{\mathcal G}}
\def\cH{{\mathcal H}}
\def\cI{{\mathcal I}}

\def\cL{{\mathcal L}}

\def\cQ{{\mathcal Q}}

\def\cS{{\mathcal S}}

\def\cU{{\mathcal U}}
\def\cV{{\mathcal V}}
\def\cW{{\mathcal W}}
\def\cX{{\mathcal X}}
\def\cY{{\mathcal Y}}

\def \sf {\mathfrak s}

\def\E{\mathsf {E}}

\def\T{\mathsf {T}}

\def\bcI{\overline \cI}
\def\tcI{\overline \cI}

\newcommand{\ignore}[1]{}

\def \balpha{\bm{\alpha}}
\def \bbeta{\bm{\beta}}

\def \C{\mathbb{C}}
\def \F{\mathbb{F}}

\def \P{\mathbb{P}}
\def \R{\mathbb{R}}

\def\rI{I}
\def\rG{G}

\def\mand{\qquad\mbox{and}\qquad}

\def\\{\cr}
\def\({\left(}
\def\){\right)}
\def\fl#1{\left\lfloor#1\right\rfloor}

\def\ep{\mbox{\bf{e}}_p}

\begin{document}


\title[ Character Sums 
with Intervals and Arbitrary Sets]{Double Character Sums 
with Intervals and Arbitrary Sets in Finite Fields}
%


\author[I. D. Shkredov]{Ilya D. Shkredov}
\address{Steklov Mathematical Institute of Russian Academy
of Sciences, ul. Gubkina 8, Moscow, Russia, 119991, and Institute for Information Transmission Problems  of Russian Academy
of Sciences, Bolshoy Karet\-ny Per. 19, Moscow, Russia, 127994,
and MIPT, Institutskii per. 9, Dolgoprudnii, Russia, 141701}
\email{ilya.shkredov@gmail.com}

\author[I. E.~Shparlinski]{Igor E.~Shparlinski}
\address{Department of Pure Mathematics, University of 
New South Wales, Sydney, NSW 2052 Australia}
\email{igor.shparlinski@unsw.edu.au}

\begin{abstract}  
We obtain a new bound on certain double sums of multiplicative characters
improving the range of  several previous results.
This improvement comes from new bounds on the number of 
collinear triples in finite fields, which is a classical object of study of additive combinatorics. 
\end{abstract}


\maketitle

\section{Introduction}

\subsection{Motivation and background}

For a prime $p$, let $\F_p$ be the finite field of $p$ elements.

Given a multiplicative character  $\chi$  of the multiplicative group $\F_p^*$ (see~\cite[Chapter~3]{IwKow} 
for a background on characters), we define the bilinear multiplicative character sums
$$
W_\chi(\cI, \cS;\balpha, \bbeta) = \sum_{s\in \cS}  \sum_{x \in \cI} \alpha_s \beta_x    \chi(s+x) \,, 
$$
where $\cI = [1,X]$ is an interval, $\cS \subseteq \F_p$ is  an arbitrary  set of cardinality $\# \cS = S$, 
and  $\balpha = \{\alpha_s\}_{s\in \cS}$ and $\bbeta = \{\beta_x\}_{x\in \cI}$  
are two sequence of complex weights with values inside of the unit disk:
\begin{equation}
\label{eq:unit disc}
|\alpha_s| \le 1, \quad s\in \cS, \mand |\beta_x| \le 1, \quad x\in \cI. 
\end{equation}
In particular, the sums 
$$
 \sum_{s\in \cS} \left| \sum_{x \in \cI}  \beta_x    \chi(s+x)
 \right| \mand 
  \sum_{x \in \cI}   \left|   \sum_{s\in \cS}   \alpha_s \chi(s+x)
 \right |
$$
are both of the same shape, and the other way around: to estimate $W_\chi(\cI, \cS;\balpha, \bbeta) $ it is enough to estimate either of these sums.

First we remark that if  
$X \ge p^{1/4+\varepsilon}$ with some  fixed  $\varepsilon>0$, in the case when of the trivial 
weights  $\beta_x =1$, $x\in \cI$, the Burgess bound implies that
$$
\sum_{s\in \cS}  \left| \sum_{x \in \cI}     \chi(s+x)\right| = O(SX p^{-\delta})
$$
for some $\delta> 0$ that depends only on $\varepsilon>0$. 

Furthermore, the result of  Karatsuba~\cite{Kar1}  (see also~\cite[Chapter~VIII, Problem~9]{Kar2})
which applies to general bilinear sums  gives a nontrivial estimate on 
$W_\chi(\cI, \cS;\balpha, \bbeta)$ when 
$$
\min\{S,X\} \ge p^{\varepsilon} \mand \max\{S,X\} \ge p^{1/2+\varepsilon} \,. 
$$

Finally, by  a result of Chang~\cite[Theorem~9]{Chang1}, there is a 
nontrivial bound 
$$
 \sum_{x \in \cI}   \left|  \sum_{s\in \cS}  \chi(s+x)\right| = O(SX p^{-\delta})
$$
provided that for some fixed real $\xi> 0 $ and $\zeta > 0$ we have $S = p^{\xi+o(1)}$, 
$X = p^{\zeta + o(1)}$ and for 
$$
k = \fl{\zeta ^{-1}}
$$
we have 
\begin{equation}
\label{eq:Chang Cond}
\xi > \frac{3k-2 - 4k \zeta }{6k-8}
\end{equation}
(where $\delta> 0$ depends only on $\xi$ and $\zeta$). 

On the other hand, the method of Karatsuba~\cite{Kar1} gives a  similar bound 
under the condition
\begin{equation}
\label{eq:Kar Cond}
\xi > \frac{1- \zeta}{2}
\end{equation}
which is worse than~\eqref{eq:Chang Cond} if $\xi$ and $\zeta$ are close to each other.
For example, for $\zeta=\xi$ the conditions~\eqref{eq:Chang Cond} and~\eqref{eq:Kar Cond}
become 
\begin{equation}
\label{eq:Diag Cond}
\zeta=\xi > 7/22 \mand \zeta=\xi>1/3 \,,
\end{equation}
respectively. 

On the other hand, the condition~\eqref{eq:Kar Cond} is better than~\eqref{eq:Chang Cond} provided that 
\begin{equation}
\label{eq:We win}
1/4 < \zeta < 2/7 \,.
\end{equation}
This indicates that the approach of Karatsuba~\cite{Kar1} is still competitive and  deserves further investigation. Because of this, and because it 
does not seem to be ever presented in full detail, we do this here. Furthermore, we complement this approach 
with some new ingredients  coming from recent advances in additive combinatorics  which 
lead to a  stronger  bound.

There are also several bounds~\cite{BGKS1, BGKS2, BKS2, Chang1, FrIw} but they 
apply only for some special sets $\cS$, such sets with well-spaced elements or sets 
that are contained in short intervals. 

Various bounds of character sums sums with more than two arbitrary sets can be found 
here~\cite{Hans,ShkShp,ShkVol}. 

\subsection{Main results}

We formulate our main result in terms of the  quantity $\E^{+}_3 (\cU, \cV, \cW)$ 
which for   sets  $\cU, \cV, \cW \subseteq\F_p$ is defined
 as the number of the solutions to the 
equation 
\begin{equation}
\begin{split}
\label{eq:E3}
  u_1 &-u_2=v_1-v_2=w_1 - w_2, \\
  u_1,u_2 & \in \cU,\ v_1,v_2 \in \cV,\ w_1,w_2 \in \cW. 
\end{split}
\end{equation}

Assuming that  $\# \cU \ge \# \cV \ge\# \cW$,  the  trivial upper bound for $\E^{+}_3 (\cU, \cV, \cW)$  is 
\begin{equation}
\label{eq:E3 trivial}
\E^{+}_3 (\cU, \cV, \cW)\le \# \cU  \# \cV  \#\cW  \min\{ \# \cU,  \# \cV ,  \#\cW\}\,.
\end{equation}

We recall that the notations $U = O(V)$,  $U \ll V$ and  $V \gg U$  are
all equivalent to the statement that $|U| \le c V$ holds
with some constant $c> 0$, which throughout this work may depend 
on the integer parameter $r \ge 1$.

\begin{theorem}
\label{thm:Gen Set E3}
  For any interval  $\cI = [1,X]$ 
of length $X$
and  any set $\cS \subseteq\F_p^*$  of size $\# \cS=S$
such that  
\begin{equation}
\label{eq:cond SX}
S^2 X \le  p^{2} \mand X < p^{1/2}
\end{equation}
and complex weights $\balpha = \{\alpha_s\}_{s\in \cS}$ and $\bbeta = \{\beta_x\}_{x\in \cI}$  
satisfying~\eqref{eq:unit disc},
for any fixed integer $r \ge 1$ such that $X \ge p^{1/r}$,  we have
\begin{align*}
& W_\chi(\cI, \cS;\balpha, \bbeta)    \\
&\qquad  \ll SX \( \frac{\E^{+}_3 (\cS, \cS, \bcI) p^{(r+1)/r}}{S^4X^3}+\frac{ p^{(r+2)/r}}{S  X^{5/2}} +\frac{p^{(r+2)/r}}{S^{2} X^{2}} \)^{1/4r} p^{o(1)} \\
& \qquad\qquad\qquad\qquad\qquad\qquad\qquad\qquad\qquad \qquad\qquad \qquad\ + S^{1/2} X\,, 
\end{align*} 
where $\bcI = [-X,X]$. 
\end{theorem}

Combining Theorem~\ref{thm:Gen Set E3} with the trivial bound~\eqref{eq:E3 trivial} we obtain:

\begin{cor}
\label{cor:Gen Set Triv}  Under the condition of Theorem~\ref{thm:Gen Set E3}, we have
\begin{align*}
& W_\chi(\cI, \cS;\balpha, \bbeta)    \\
&\qquad \quad  \ll SX \( \frac{M  p^{(r+1)/r}}{S^2X^2}+\frac{ p^{(r+2)/r}}{S  X^{5/2}} +\frac{p^{(r+2)/r}}{S^{2} X^{2}} \)^{1/4r} p^{o(1)} + S^{1/2} X\,, 
\end{align*}
where $M = \min\{S,X\}$.
\end{cor}

The result of Chang~\cite[Theorem~9]{Chang1} is not fully explicit so it may not be straightforward 
to compare it with Corollary~\ref{cor:Gen Set Triv}, however, see~\eqref{eq:We win} for the range of parameters, where 
Corollary~\ref{cor:Gen Set Triv} is certainly stronger. We also note that 
compared to the original method of  Karatsuba~\cite{Kar1},  our main technical innovation 
is Lemma~\ref{lem:TABC} which can be of independent interest (and thus we present 
in a more general form than we need in this work), see also Remark~\ref{rem:Kar bound}. 

Furthermore, given an interval $\cI = [1, X]$, we have the trivial inequality
\begin{equation}
\label{eq:E3 E2X}
\E^{+}_3 (\cS, \cS, \cI) \ll X\E^{+}(\cS)
\end{equation}
where $\E^{+}(\cS)$ is the {\it additive energy\/} of the set $\cS$, that is, 
$$
\E^{+}(\cS) =\#\{ u_1 +u_2=v_1+v_2~:~  u_1,u_2,v_1,v_2 \in \cS\}\,.
$$ 
Thus we have the following bound
which underlines one of the bounds used in the proof of Theorem~\ref{thm:Subg}
(see the proof of Lemma~\ref{lem:E3 subgr} below). 

\begin{cor}
\label{cor:Gen Set Energy}  Under the condition of Theorem~\ref{thm:Gen Set E3}, we have
\begin{align*}
& W_\chi(\cI, \cS;\balpha, \bbeta)    \\
&\qquad  \ll SX \( \frac{\E^{+} (\cS) p^{(r+1)/r}}{S^4X^2}+\frac{ p^{(r+2)/r}}{S  X^{5/2}} +\frac{p^{(r+2)/r}}{S^{2} X^{2}} \)^{1/4r} p^{o(1)} \\
& \qquad\qquad\qquad\qquad\qquad\qquad\qquad\qquad\qquad \qquad\qquad \qquad\ + S^{1/2} X\,, 
\end{align*}
\end{cor}

If $\E^{+}_3 (\cS, \cS, \bcI) \gg S^3 X$, then by~\eqref{eq:E3 E2X}  
the additive energy $\E^{+} (\cS)$ of $\cS$ is large and 
thus $\cS$ is very structured.   
On the other hand, one can easily give examples of sets with small 
quantity $\E^{+}_3 (\cS, \cS, \bcI)$,  for example, 
for random sets $\cS$ or for sets with some prescribed algebraic structure.

Here we give one of such examples, namely when $\cS$ is a multiplicative subgroup $\cG\subseteq \F_p$
for which the sums $W_\chi(\cI, \cS;\balpha, \bbeta)$ have been considered in~\cite{ChangShp,ShpYau}
(in the case of constant weights $\balpha$ and $\bbeta$). 
To simplify the exposition, and enable us to apply a result of Cilleruelo and  Garaev~\cite[Theorem~1]{CillGar}
 we always assume that $\# \cG \le p^{2/5}$. Note that for large subgroups one can use the results and
methods of~\cite{ChangShp,ShpYau}. 
 
\begin{theorem}
\label{thm:Subg} 
For any fixed positive $\zeta < 1/2$ and  $\xi< 2/5$ satisfying
 $$
\xi > \left\{
 \begin{array}{ll} 
 1-5 \zeta/2, & \text{if}\ 6/25 <  \zeta < 10/31\,,\\
 (6-9 \zeta)/16, & \text{if}\ 10/31 \le  \zeta <  134/361\,,\\
(20-40 \zeta)/31, & \text{if}\     134/361  \le  \zeta <  1/2 \, .
 \end{array}
 \right.
$$
there exists some $\delta>0$ such that for any interval  $\cI = [1,X]$ 
of length $X  = p^{\zeta + o(1)}$ and  a multiplicative subgroup $\cG\subseteq \F_p$
of order $T = p^{\xi+o(1)}$,   and complex weights 
$\balpha = \{\alpha_s\}_{s\in \cG}$ and $\bbeta = \{\beta_x\}_{x\in \cI}$  satisfying~\eqref{eq:unit disc},
we have
$$
 W_\chi(\cI, \cG;\balpha, \bbeta)  \ll TX  p^{-\delta}.
$$
\end{theorem}

We note that the bound of~\cite{ShpYau} is nontrivial only under the 
condition~\eqref{eq:Kar Cond}. 
Furthermore, if $\zeta=\xi$, the bound of Theorem~\ref{thm:Subg}  is nontrivial 
for 
$$
\zeta=\xi > 2/7
$$
improving on what one can derive from~\eqref{eq:Diag Cond}.

\begin{rem}
\label{rem:E3 shift} 
{\rm
Clearly  the quantity $\E^{+}_3 (\cS, \cS, \bcI)$ is invariant under  translations $\cS \to \cS +a$ of $\cS$ 
by $a \in \F_p$.
Thus the bound of Theorem~\ref{thm:Subg} also holds for the sums 
$W_\chi(\cI, \cG;\balpha, \bbeta)$ where $\cI = [a+1, a+X]$ is an arbitrary interval. 
}
\end{rem}

\begin{rem}
\label{rem:E3 vs E2} 
{\rm 
The additive energy of sets plays a central role in additive combinatorics, see~\cite{TaoVu}, 
and has been studied in a vast number of works. For example, using the bound of 
Corvaja and Zannier~\cite[Theorem~2]{CoZa} one derives
$$
\E^{+}_3 \(f(\cG), f(\cG), \bcI\) \ll T^{8/3} X
$$
for a polynomial image $f(\cG)$ of  a multiplicative subgroup $\cG\subseteq \F_p$
of order $T \le  p^{3/4})$ and a polynomial $f(Z) \in \F_p[Z]$ (under some mild conditions on $f$). Together with Theorem~\ref{thm:Gen Set E3} this leads to
new bounds on  the sums 
$W_\chi(\cI,  f(\cG);\balpha, \bbeta)$ complementing the bound of~\cite[Theorem~1.2]{ChangShp}.
}
\end{rem}

Similarly to polynomial images of subgroups in Remark~\ref{rem:E3 vs E2}, 
one can use  bounds on  the additive 
energy of polynomial images of intervals in a combination with Theorem~\ref{thm:Gen Set E3}.
We given only two very concrete applications of this type to character sums over primes. 

\begin{theorem}
\label{thm:Primes} Let $f$ be a  polynomial  over $\F_p$ of degree $d\ge  2$. 
For any $Q=p^{\zeta + o(1)}$ and $R=p^{\xi + o(1)}$ an fixed positive $\zeta$ and  $ \xi\le \min\{1/2, 2 - 2 \zeta\}$ satisfying
$$
5\zeta/4+2\xi > 1 \mand \zeta+5\xi/2 > 1,
$$
for $d=2$ and 
$$
\(1 + 2^{-d+1}\) \zeta+2\xi > 1 \mand \zeta+5\xi/2 > 1,
$$
for $d \ge 3$, 
there exists some $\delta>0$ such that 
$$
  \sum_{\substack{q\le Q\\q~\text{prime}}}  \left|  \sum_{\substack{r \le R\\ r~\text{prime}}}
 \chi(f(q) +r)\right|\,, \  \sum_{\substack{r \le R\\ r~\text{prime}}}  \left| \sum_{\substack{q\le Q\\ q~\text{prime}}} 
 \chi(f(q) +r )\right|  
 \ll QR   p^{-\delta}\,. 
$$
\end{theorem}

 \begin{rem}
\label{rem:Poly Energy} 
{\rm The bound of Theorem~\ref{thm:Primes} makes use  of a bound on  the additive energy of 
polynomial images given in Corollary~\ref{cor:pol_images energy}, which in turn relies on 
a new result from additive combinatorics given in Lemma~\ref{lem:pol_images}.  Furthermore,
 Lemma~\ref{lem:pol_images} implies an improvement of a result of 
 Bukh and Tsimerman~\cite[Theorem~1]{BT}, see Remark~\ref{rem:Imrove BT}. 
 On the other hand, it seems that for polynomials of high degree the approaches of~\cite{Chang2,CCGHSZ,CGOS} are likely to  become 
more efficient.  
}
\end{rem}

\section{Background from Additive Combinatorics}

\subsection{Points--planes incidences}

Now we recall some notions about points--planes incidences, in which we follow~\cite{Shkr}.

First of all, we need a general design bound for the number of incidences.  
Let $\cQ \subseteq \F_p^3$ be a set of points  and $ \Pi$ be a collection of planes in $\F_p^3$. 
Having $q \in\cQ$ and $\pi \in \Pi$ we write 
\begin{displaymath}
\rI (q,\pi) = \left\{ \begin{array}{ll}
1 & \textrm{if } q\in \pi,\\
0 & \textrm{otherwise.}
\end{array} \right.
\end{displaymath}
So, $\rI$ is a $(\# \cQ \times \# \Pi)$-matrix.

 If $ \cQ = \F_p^3$ and $\Pi$ is the family of all planes in $\F_p^3$, then we obtain the matrix $\rG$ and thus $\rI$ is a submatrix of $\rG$. 
One can easily calculate $\rG^t \rG$ and $\rG\, \rG^t$ (where $\rG^t$ 
is the transposition of $\rG$) 
embedding $\F_p^3$ into the projective space $\P\F_p^3$
and check that both of these matrices are of the 
form 
$a\mathbf{Id} + b \mathbf{1}$, where $a,b$ are some scalar coefficients, $\mathbf{Id}$ and $\mathbf{1}$ and identity matrix and all-ones matrices of corresponding dimensions, see, for example,~\cite{TTT,Vinh}.
Moreover, one can check that in our case of points and planes  the following
holds  $a = p^2$ and $b= p+1$ (see~\cite{TTT,Vinh}).
In other words, $G G^t = p^2 \mathbf{Id} + (p+1) \mathbf{1}$.
In view of  
these facts and using the singular decomposition 
(see, for example,~\cite{HoJo}), denoting 
$$
 Q = \# \cQ \,,
$$
we see that
$$
	G (q,\pi) = \sum_{j=1}^{Q} \mu_j u_j (q) v_j (\pi) \,,
$$
where $\mu_j \ge 0$ are square--roots of the eigenvalues of $GG^t$ (which coincide with  square--roots of nonzero eigenvalues of $G^t G$)
and  $u_j$ and $v_j$,   are the  eigenfunctions of $GG^t$ and $G^t G$, 
respectively, $j =1, \ldots, Q$. 
From $GG^t = p^2 \mathbf{Id} + (p+1) \mathbf{1}$, we obtain 
$$
\mu^2_1 = p^2 + (p+1) Q \mand \mu_2 = \ldots = \mu_{Q} = p
$$ 
and  
$$
u_1 (q) = (1,\dots,1) \in \R^{Q}  \mand v_1 (\pi) = (1,\dots,1) \in \R^{P}  \,, 
$$
where $P = \# \Pi$. 
Hence we 
derive 
that 
for any functions $f :  \cQ \to \C$ and $g: \Pi \to \C$, supported only $\cQ$ and $\Pi$, 
respectively, one has 
\begin{align*}
\left|\sum_{q \in \cQ} \sum_{\pi \in \Pi} \rI (q,\pi) f(q) g(\pi) \right| 
& = \left|\sum_{q \in \cQ} \sum_{\pi \in \Pi} \rG (q,\pi) f(q) g(\pi)\right| \\
& = \left| \sum_{j=2}^{Q} \mu_j \langle f, u_j \rangle \langle g, v_j \rangle  \right| \\
& \le p \sum_{j=2}^{Q} |\langle f, u_j \rangle \langle g, v_j \rangle|  \,, 
\end{align*}
provided  that
\begin{equation}
\label{eq:cond Vinh fg}
\sum_{q\in  \cQ} f(q) = 0 \quad \text{or} \quad  \sum_{\pi\in \Pi}g (\pi) = 0 \,.
\end{equation}
Using the Cauchy inequality we now  see that under the condition~\eqref{eq:cond Vinh fg} we have 
\begin{equation}\label{f:Vinh}
\left|\sum_{q \in \cQ} \sum_{\pi \in \Pi} \rI (q,\pi) f(q) g(\pi) \right|  \le p \| f \|_2 \| g \|_2 \,,
\end{equation}
where, as usual  $\|f \|_2$ and $\|g \|_2$ are the  $L^2$-norms of 
functions $f$ and $g$, respectively. 

Furthermore, a  deep 
result on incidences of Rudnev~\cite{Rud} (or see~\cite[Theorem~8]{MPR-NRS} and the proof of~\cite[ Corollary~2]{MuPet2}) combined with the incidence bound 
from~\cite[Section~3]{MuPet1} leads to the following asymptotic formula:

\begin{lemma}
\label{lem:Misha+}
Let  $\cQ \subseteq \F_p^3$ be a set of points and let $\Pi$ be a collection of planes in $\F_p^3$. Suppose that $\# \cQ \le \# \Pi$ and that $k$ is the maximum number of collinear points in $\cQ$. Then
$$
\sum_{q \in \cQ} \sum_{\pi \in \Pi} \rI (q,\pi)  - \frac{\# \cQ \# \Pi}{p} \ll \(\# \cQ\)^{1/2} \# \Pi + k \# \cQ \,.
$$
\end{lemma}

\subsection{On the number of collinear triples}
\label{sec:T(A)}

Given three sets $\cA,\cB,\cC \subseteq \F_p$ we  denote by $\T(\cA,\cB,\cC)$ 
the number of the solutions to the equation
$$
\frac{a_1 - c_1}{b_1-c_1} =\frac{a_2 - c_2}{b_2-c_2} , \qquad   
a_1,a_2\in \cA,\, b_1,b_2 \in B,\, c_1,c_2 \in \cC \,. 
$$
Geometrically, $\T(\cA,\cB,\cC)$ is the number of collinear triples of points 
$$
\((a_1,a_2), \(b_1,b_2\),  (c_1,c_2)\) \in \cA^2\times \cB^2 \times \cC^2 \,.
$$

Let $\cL$ be the set of all lines in $ \F_p^2$. 
We observe that there are exactly $p+1$ lines passing via any 
point in  $ \F_p^2$.

Given $\ell \in \cL$, we denote
$$
\iota_\cA (\ell) = \#\(\ell \cap \cA^2\). 
$$

We denote by $\ \cL^*\(\cA,\cB,\cC\)$ the set of lines $\ell$ having at least two points from 
$(\cA\times \cA) \cup (\cB\times \cB) \cup (\cC\times \cC)$. 
Then we see that 
$$
\T(\cA,\cB,\cC) = \sum_{\ell \in \cL^*\(\cA,\cB,\cC\)} 
\iota_\cA (\ell) \iota_\cB (\ell)  \iota_\cC (\ell)  \,.
$$

Finally, for any real $M>0$ put
$$
L_\cA (M) = \{\ell~:~ M< \iota_\cA(\ell)  \le 2M \} \,.
$$

We need~\cite[Lemma~14]{MPR-NRS}.

\begin{lemma}
\label{lem:L_M}
Let $\cA\subseteq \F_p$ be a set and let $M$ be a real number with 
 $\# \cA \ge M \ge 2\(\# \cA\)^2/p$, then
$$
L_\cA (M) \ll \min \left\{ \frac{p\(\# \cA\)^2}{M^2}, \frac{\(\# \cA\)^5}{M^4} \right\} \,.
$$
\end{lemma}

First we record the trivial identity 
\begin{equation}
\label{eq:i_A}
\sum_{\ell \in \cL} \iota_\cA (\ell)   =  (p+1) \(\# \cA\)^2 \,,
\end{equation}
which holds for any set $\cA \subseteq \F_p$, (as there are exactly $p+1$ lines 
passing through  any   point $(a_1,a_2) \in \F_p^2$).

The next  identity is well--known however,  we give a short  proof for the sake of the completeness.

\begin{lemma}
\label{lem:i_AB}
For $\cA, \cB \subseteq \F_p$, we have
$$
\sum_{\ell \in \cL} \iota_\cA (\ell) \iota_\cB (\ell)  =\(\# \cA \#\cB\)^2 + p \#\( \cA^2\cap\cB^2\) \,.
$$
\end{lemma}

\begin{proof}
We have
\begin{align*}
\sum_{\ell \in \cL} \iota_\cA (\ell) \iota_\cB (\ell)  &= \sum_{\ell \in \cL} 
\sum_{\substack{q \in \cA^2\\ q \in \ell}}  
\sum_{\substack{r \in \cB^2\\ r \in \ell}} 1\\
& = \sum_{\substack{(q,r) \in \cA^2\times  \cB^2\\ q \neq r}}\, \sum_{\substack{\ell \in \cL\\ q,r \in \ell}} 1
+\sum_{q \in\cA^2\cap\cB^2 }\,\sum_{\substack{\ell \in \cL\\ q \in \ell}} 1 \,.
\end{align*}

Clearly, since two distinct points define a unique line, we have
$$
\sum_{\substack{(q,r) \in \cA^2\times  \cB^2\\ q \neq r}}\, 
\sum_{\substack{\ell \in \cL\\ q,r \in \ell}} 1
= \(\# \cA \#\cB\)^2 - \#\(\cA^2\cap\cB^2\) \,.
$$
Furthermore, using again that there are exactly $p+1$ lines passing via any 
point in  $ \F_p^2$ we also have
$$
\sum_{q \in \cA^2\cap\cB^2 }\, 
\sum_{\substack{\ell \in \cL\\ q \in \ell}} 1 
= (p+1)\#\(\cA^2\cap\cB^2\) \,.
$$
The result now follows.  
\end{proof}

For a set $\cA  \subseteq \F_p$ we now define the 
function
\begin{equation}
\label{eq:fA def}
f_\cA (\ell)  = \iota_\cA (\ell) - \frac{\(\# \cA\)^2}{p}  \,.
\end{equation}
In particular, we see from Lemma~\ref{lem:i_AB} that
\begin{equation}
\label{f:i_scalar}
\sum_{\ell \in \cL} | f_\cA (\ell)|^2 \le p \(\# \cA\)^2 \,.
\end{equation}

Finally, for any real  $M > 0$ put
$$
K_\cA (M) = \{\ell~:~ |f_\cA(\ell)| > M \} \,.
$$

\begin{lemma}
\label{lem:K_M}
Let $\cA\subseteq \F_p$ be a set and let $M$ be a real number, then
$$
K_\cA (M) \ll \min \left\{ \frac{p\(\# \cA\)^2}{M^2}, \frac{\(\# \cA\)^5}{M^4} \right\} \,.
$$

\end{lemma}

\begin{proof} For $M \ge  2 \(\# \cA\)^2/p$ the result follows from  Lemma~\ref{lem:L_M} as in this case $|f_\cA(\ell)| > M$ implies 
$$
2|f_\cA(\ell)|   \ge \iota_\cA(\ell) \ge  |f_\cA(\ell)|/2 \,. 
$$

For $M <   2\(\# \cA\)^2/p$ we derive from~\eqref{f:i_scalar}
$$
K_\cA (M) \ll \frac{p\(\# \cA\)^2}{M^2} \,.
$$
Since  for $M  <   2\(\# \cA\)^2/p \le 2\(\# \cA\)^{3/2}/p^{1/2}$
we have
$$
\frac{p\(\# \cA\)^2}{M^2} \ll   \frac{\(\# \cA\)^5}{M^4} 
$$
and the result follows. 
\end{proof}

 Finally, we are ready to establish one of our main technical result which 
 we believe is of independent interest. 
 
\begin{lemma}
\label{lem:TABC}
Let $\cA,\cB,\cC \subseteq \F_p$ be sets. 
Put $Z=\max\{ \# \cA, \# \cB, \#\cC \}$. Then
$$
\T(\cA,\cB,\cC) - \frac{\(\# \cA \#\cB \# \cC\)^2}{p} 
\ll \begin{cases}
p \# \cA \#\cB \# \cC \,,\\ 
\(\# \cA \#\cB \# \cC\)^{3/2} +  \# \cA \#\cB \# \cC Z\,  \\
 \sqrt{p} \(\# \cA \#\cB \# \cC\)^{7/6} + Z^4\,.
\end{cases}
$$
\end{lemma}

\begin{proof}
We  basically repeat the arguments from~\cite{Shkr}. 

Put 
$$
g_\cA(x) = \chi_\cA(x) - \frac{\# \cA}{p},
$$
where $\chi_\cA(x)$ is the 
characteristic function of the set $\cA$. Thus
\begin{equation}
\label{eq:cond Vinh gA}
\sum_{x\in \F_p} g_\cA(x) = 0 \,.
\end{equation}
It is easy to see~\cite{SZ_inc} that the quantity $\T(\cA,\cB,\cC)$  equals the number of incidences  between the planes
\begin{equation}\label{tmp:20.12_1}
\frac{1}{b_1-c_1} X - a_2Y +Z = -\frac{c_1}{b_1-c_1} 
\end{equation}
and the points 
$$
\(a_1, \frac{1}{b_2-c_2}, \frac{c_2}{b_2-c_2} \) 
$$
with $a_1,a_2\in \cA$, $b_1,b_2 \in B$,  $c_1,c_2 \in \cC$.  

Using the function $g_\cA(x)$, we have
\begin{equation}
\label{eq:T sigma}
\T(\cA,\cB,\cC) = \frac{\(\# \cA \#\cB \# \cC\)^2}{p}  + \sigma \,,
\end{equation}
where the sum $\sigma$ 
counts the number of incidences~\eqref{tmp:20.12_1}
with the weight $g_\cA(a_1) g_\cA(a_2)$. 
Hence by~\eqref{f:Vinh},  which implies as we see from~\eqref{eq:cond Vinh gA}
that the condition~\eqref{eq:cond Vinh fg}
is satisfied,   we get 
\begin{equation}
\label{eq:bound1fin}
|\sigma| \le p \# \cA \#\cB \# \cC
\end{equation}
and by Lemma~\ref{lem:Misha+}, we have 
\begin{equation}
\label{eq:bound2}
\sigma \ll \frac{\(\# \cA \#\cB \# \cC\)^2}{p} + \(\# \cA \#\cB \# \cC\)^{3/2} + \# \cA \#\cB \# \cC Z \,.
\end{equation}
We now observe that if $\# \cA \#\cB \# \cC\ge p^2 $, then   
$$
 p \# \cA \#\cB \# \cC\le \(\# \cA \#\cB \# \cC\)^{3/2} \,.
$$
and thus the bound~\eqref{eq:bound1fin} is stronger than~\eqref{eq:bound2}. 
On the other hand, for $\# \cA \#\cB \# \cC< p^2$, the  first term in 
the bound~\eqref{eq:bound2} never dominates and it 
simplifies as 
\begin{equation}
\label{eq:bound2fin}
\sigma \ll 
\(\# \cA \#\cB \# \cC\)^{3/2} +  \# \cA \#\cB \# \cC Z \,.
\end{equation}

Now we   always use $\vartheta_j$ to denote some real numbers 
with $|\vartheta_j| \le 1$, $j=1, 2, \ldots$. 

Let $f_\cA (\ell)$  be defined by~\eqref{eq:fA def}, we also 
define   $f_\cB (\ell)$ and  $f_\cC (\ell)$  similarly. 

We have
\begin{align*}
\T(\cA,\cB,\cC) &= \sum_{\ell \in \cL} \iota_\cA (\ell) \iota_\cB (\ell)  \iota_\cC (\ell) \\
& =  \sum_{\ell \in \cL} f_\cA(\ell)  \iota_\cB (\ell)  \iota_\cC (\ell)   + \frac{\(\# \cA\)^2}{p} \sum_{\ell \in \cL} \iota_\cB (\ell)  \iota_\cC (\ell) \,.
\end{align*}
Hence, by Lemma~\ref{lem:i_AB}, estimating  $\#\( \cB^2\cap\cC^2\)$ trivially as $Z^2$, we obtain
$$
\T(\cA,\cB,\cC)  =
\sum_{\ell \in \cL} f_\cA(\ell)  \iota_\cB (\ell)  \iota_\cC (\ell)   + \frac{\(\# \cA \#\cB \# \cC\)^2}{p} + \vartheta_1 \(\# \cA\)^2 Z^2 \,.
$$

Thus, also using $\# \cA \le Z$ for $\sigma$ defined by~\eqref{eq:T sigma}, we obtain
\begin{equation}
\label{eq:bound3-1}
\sigma = 
\sum_{\ell \in \cL} f_\cA(\ell)  \iota_\cB (\ell)  \iota_\cC (\ell)   +   \vartheta_2 Z^4 \,.
\end{equation}

Furthermore, from the definition of $  f_\cB (\ell)$, 
\begin{equation}
\begin{split}
\label{eq:fii}
\sum_{\ell \in \cL} & f_\cA(\ell)  \iota_\cB (\ell)  \iota_\cC (\ell)   \\
 & =  \sum_{\ell \in \cL} f_A(\ell)  f_\cB (\ell)  \iota_\cC (\ell)  +  \frac{\(\# \cB\)^2}{p} \sum_{\ell \in \cL} f_\cA (\ell) \iota_\cC (\ell) \,. 
 \end{split}
\end{equation}
Using~\eqref{eq:i_A} and then Lemma~\ref{lem:i_AB} again, together with  
trivial bound $\#\( \cB^2\cap\cC^2\) \le Z^2$, we see that 
\begin{align*}
 \sum_{\ell \in \cL} f_\cA (\ell) \iota_\cC (\ell) &
= 
 \(\# \cA\#\cC \)^2 +   \vartheta_2 p Z^2 - \frac{\(\# \cA\#\cC\)^2 (p+1)}{p} \\
 &=  - \frac{\(\# \cA\#\cC\)^2}{p} + \vartheta_2p Z^2 \,, 
\end{align*}
where in fact $\vartheta_2 \in [0,1]$. 
Clearly, 
$$
\frac{\(\# \cA\#\cC\)^2}{p} \le p Z^2\,.
$$
Hence 
\begin{equation}
\label{eq:Z4}
  \frac{\(\# \cB\)^2}{p} \sum_{\ell \in \cL} f_\cA (\ell) \iota_\cC (\ell)   =   \vartheta_3  Z^4\,,
\end{equation}
which after substitution in~\eqref{eq:fii} and then in~\eqref{eq:bound3-1} yields
\begin{equation}
\label{eq:bound3-2}
\sigma =
 \sum_{\ell \in \cL} f_A(\ell)  f_\cB (\ell)  \iota_\cC (\ell)   +  2 \vartheta_4 Z^4\,.
\end{equation}
Finally, the same arguments as in the above and an analogue of~\eqref{eq:Z4}, 
lead us to the
bound 
\begin{align*}
 \sum_{\ell \in \cL} f_A(\ell)  f_\cB (\ell)  \iota_\cC (\ell) &  = 
 \sum_{\ell \in \cL} f_A(\ell)  f_\cB (\ell) f_\cC (\ell) +  \frac{\(\# \cC\)^2}{p} 
    \sum_{\ell \in \cL} f_A(\ell)  f_\cB (\ell) \\
 & = \sum_{\ell \in \cL} f_A(\ell)  f_\cB (\ell) f_\cC (\ell) +  \ \vartheta_5  Z^4\,.
 \end{align*}
 Now, recalling~\eqref{eq:bound3-2}, we obtain
\begin{equation}
\label{eq:bound3-3}
\sigma = \sigma_0  +  3 \vartheta_6 Z^4\,, 
\end{equation}
where 
$$
\sigma_0 = \sum_{\ell \in \cL} f_A(\ell)  f_\cB (\ell) f_\cC (\ell) \,.
$$

It remains to estimate $\sigma_0$.

We fix some real numbers $\Delta_\cA, \Delta_\cB, \Delta_\cC> 0$ and write 
\begin{align*}
\sigma_0 & \ll p\(\Delta_\cA  \# \cB \# \cC + \Delta_\cB  \# \cA \# \cC + \Delta_\cC \# \cA  \# \cB\) \\
& \qquad \qquad \qquad +
\sum_{\substack{\ell \in \cL\\ |f_\cU(\ell) | \ge \Delta_\cU,\ \cU = \cA, \cB, \cC}} \left|f_\cA(\ell)  f_\cB (\ell)  f_\cC (\ell)\right|\, .
 \end{align*}
By the H\"older inequality  
$$
	\sum_{\substack{\ell \in \cL\\ |f_\cU(\ell) | \ge \Delta_\cU,\ \cU = \cA, \cB, \cC}} \left|f_\cA(\ell)  f_\cB (\ell)  f_\cC (\ell)\right| 
	\le
	\prod_{\cU = \cA, \cB, \cC}  \left( \sum_{\substack{\ell \in \cL\\ |f_\cU(\ell) | \ge \Delta_\cU}} \left|f_\cA(\ell) \right|^3 \right)^{1/3} \,.
$$

 Given $F>0$ we note that by  Lemma~\ref{lem:K_M} we have 
 $$
\sum_{\substack{\ell \in \cL\\  F\le \left|f_\cA(\ell)\right| \le 2F}} \left|f_\cU(\ell) \right|^3 \ll \frac{F^3 (\# \cA)^5}{F^4}  
=\frac{ (\# \cA)^5}{F} \,.
$$
It follows that 
$$
\sum_{\substack{\ell \in \cL\\ |f_\cU(\ell) | \ge \Delta_\cU}} \left|f_\cA(\ell) \right|^3 
	\ll\frac{ (\# \cA)^5}{\Delta_\cU} \,.
$$
Hence 
$$
\sigma_0 \ll p(\Delta_\cA  \# \cB  \# \cC + \Delta_\cB \# \cA  \# \cC + \Delta_\cC \# \cA \# \cB)  
+
\frac{\# \cA^{5/3}  \# \cB^{5/3}  \# \cC^{5/3}}{(\Delta_\cA \Delta_\cB \Delta_\cC)^{1/3}} \,. 
$$

We now choose 
$$
\Delta_\cU = \# \cU  \frac{(\# \cA  \# \cB  \# \cC)^{1/6}}{ p^{1/2}}, \qquad \cU = \cA, \cB, \cC\,, 
$$ 
and  obtain
$$
\sigma_0 \ll \sqrt{p} \(\# \cA \#\cB \# \cC\)^{7/6} 
$$
which together with~\eqref{eq:bound3-3} implies
\begin{equation}
\label{eq:bound3fin}
\sigma \ll \sqrt{p} \(\# \cA \#\cB \# \cC\)^{7/6}  +    Z^4\,.
\end{equation}

Combining the bounds~\eqref{eq:bound1fin}, \eqref{eq:bound2fin}
and~\eqref{eq:bound3fin} with~\eqref{eq:T sigma}, we conclude the proof. 
\end{proof}
 
\subsection{Bounds on the number of solutions to some equations with general sets}
\label{sec:number sols gen set}

For   sets  $\cS, \cX, \cY\subseteq\F^*_p$,  we define $N(\cS, \cX,\cY)$ to be the number of solutions to the system of equations:
\begin{equation}
\label{eq:syst xsxtyy}
\begin{split}
&\frac{x_1+s_1}{y_1} = \frac{x_2+s_2}{y_2}   \mand 
\frac{x_1+t_1}{y_1} = \frac{x_2+t_2}{y_2}, \\
&\ \,  s_1,s_2, t_1, t_2 \in \cS, \ s_1 \ne t_1, \ s_2 \ne t_2, \ x_1,x_2 \in \cX,  \  y_1, y_2 \in \cY. 
\end{split}
\end{equation}

We also recall the definition~\eqref{eq:E3}. 

\begin{lemma} 
\label{lem:Bound NSXY}
Let  $\cS, \cX \subseteq \F^*_p$ be arbitrary sets of cardinalities $S$ and $X \le p^{1/2}$
such that $S^2 X \le p^2$, 
 and let $\cY$ be the set of primes  of the interval $[1,Y]$  for some $Y \le p^{1/2}$. Then we have
$$
N(\cS, \cX,\cY)\ll Y \E^{+}_3 (\cS, \cS, \cX)
+S^3 X^{3/2}+S^2 X^2\,.  
$$
\end{lemma}

\begin{proof}  We derive from~\eqref{eq:syst xsxtyy} that 
\begin{equation}
\label{eq:syst stxxyy}
\frac{x_1+s_1}{x_2+s_2} = \frac{x_1+t_1}{x_2+t_2} =  \frac{y_1}{y_2}  \ne 0.
\end{equation}

First we consider the case when the common value  $\lambda$  of every ratio in the equation~\eqref{eq:syst stxxyy}
is  $\lambda =1$.  In this case we derive   $s_1 - s_2 = t_1-t_2 = x_2-x_1$. Thus the vector
$(s_1, s_2 , t_1,t_2 , x_1, x_2) \in \cS^4\times \cX^2$ can be chosen in 
$\E^{+}_3 (\cS, \cS, \cX)$ ways. So we see from~\eqref{eq:syst stxxyy} (and  discarding the conditions that $y_1$ and $y_2$ are primes) 
that such vectors  contribute
in total  
$$
N_1 = Y\E^{+}_3 (\cS, \cS, \cX)
$$ 
to $N(\cS, \cX,\cY)$. 

By the second inequality of Lemma~\ref{lem:TABC}, using that $S^2 X \le p^2$, we see that the equation 
\begin{equation}
\label{eq:syst xsxt}
\frac{x_1+s_1}{x_2+s_2} = \frac{x_1+t_1}{x_2+t_2},  \quad  s_1,s_2, t_1, t_2 \in \cS, \ x_1,x_2 \in \cX \,,
\end{equation}
has $O(S^3 X^{3/2}+S^2 XZ)$ solutions, where $Z = \max\{S,X\}$. 
Using $Z \le S+X$ we derive 
\begin{align*}
S^3 X^{3/2}+S^2 XZ&\le S^3 X^{3/2}+S^2 X(S+X) \\
& = 
S^3 X^{3/2}+S^3 X +S^2X^2 \ll  S^3 X^{3/2}+S^2X^2 \,.
\end{align*}

However now we consider only solutions  when the common value $\lambda$ of both sides of 
the equation~\eqref{eq:syst xsxt}
satisfies $\lambda \ne 1$. For each $\lambda \ne 1$ we get an equation
of the type $y_1 = \lambda y_2$ which   has at most $1$ solutions  in primes  $y_1,y_2 \le p^{1/2}$.  
Hence  such vectors  contribute
in total  
$$
N_{\ne 1} \ll S^3 X^{3/2}+S^2X^2
$$ 
to $N(\cS, \cX,\cY)$. 

Collecting contributions $N_1$ and $N_{\ne 1}$ of both types,  that is, writing
$N(\cS, \cX,\cY) = N_1 + N_{\ne 1}$, we obtain the result.
\end{proof}

\begin{rem}
\label{rem:Kar bound} 
{\rm
We note that there are no details how this quantity $N(\cS, \cX,\cY)$ is estimated in~\cite{Kar1}. However,
judging by the range where the main result of~\cite{Kar1} is nontrivial it seems that instead of Lemma~\ref{lem:TABC} 
the number of solutions to the equation~\eqref{eq:syst xsxt} is estimated   in~\cite{Kar1} trivially as $O(S^3 X^2)$. 
}
\end{rem}

\begin{rem}
\label{rem:SX-cond} 
{\rm
The condition $S^2 X \le p^2$ of Lemma~\ref{lem:Bound NSXY} is imposed so that 
the second inequality of Lemma~\ref{lem:TABC} simplifies. One can certainly drop this restriction 
and apply  Lemma~\ref{lem:TABC}  in full generality. In turn this may lead to some alternative forms of 
 Theorem~\ref{thm:Gen Set E3}. 
 }
\end{rem}

\subsection{Bounds on the number of solutions to some equations with multiplicative subgroups}
\label{sec:number sols subgr}

Here we give some bound on $\E^{+}_3 (\cG, \cG, \cI)$ with a multiplicative subgroup 
$\cG \subseteq \F_p^*$. 

We always assume that $\cG$ is of order $T \le p^{2/5}$ as otherwise
other, more standard methods work better. 

\begin{lemma} 
\label{lem:E3 subgr}
Let  $\cG \subseteq \F^*_p$ be a multiplicative subgroup 
 of order $T \le p^{2/5}$ and $\cI = [1,X]$. Then we have
\begin{align*}
 \E^{+}_3 (\cG, & \cG, \cI) \\&\le p^{o(1)} \left\{
 \begin{array}{l}  T^{49/20} X\, ,\\
 T^2 X +  T^{4/3} X^{3/2} + T^{11/6} X^2 p^{-1/2} + T^{41/24} X^{3/2}p^{-1/8}\, .
 \end{array}
 \right.
\end{align*}
\end{lemma}

\begin{proof} By~\cite{MRSS} we have 
$$
\E^{+}(\cS) = \#\{(u_1,u_2 ,v_1,v_2)\in \cG^4~:~u_1+u_2 = v_1+v_2\} \le T^{49/20} p^{o(1)}\,,
$$
which together with~\eqref{eq:E3 E2X} immediately implies the first bound (see also Remark~\ref{rem:E3 vs E2}). 

We now derive the second bound. 
Given a set $\cA  \subseteq\F_p$, we write $r_{-}(\cA;x)$ and $r_{/}(\cA;x)$ for the number of ways $x \in \F_p$ can be expressed as a  sum $a-b$ and $a/b$ with $a,b\in \cA$, 
respectively. 
  
Put 
$$
\tcI = \cI-\cI = \{x-y~:~  x,y \in \cI\} = [-X,X]\, .
$$
Then 
\begin{equation}
\label{eq:E R} 
\E^{+}_3 (\cG, \cG, \cI) \le T^2 X + 4 R X \,,
\end{equation}
where 
$$
R =   \sum_{x \in \cI^*,\, x\neq 0} r^2_{-} (\cG;x) .
$$ 
We note that for $x \in \tcI$ the value of $ r_{-} (\cG;x)$ depends only 
on the coset $\lambda\cG$ with $x \in \lambda\cG$. 

Let $h = (p-1)/T$. 

Consider cosets $\cC_j =\lambda_j \cG$, $j=1, \ldots, h$,  
such that for $x_j \in \cC_j$ one has 
$$
r_{-} (\cG;x_1) \ge \ldots\ge  r_{-} (\cG;x_h) \,.
$$
Let  
$$
c_j =  \#\(\cC_j \cap \tcI\) \mand t_j = r_{-} (\cG;x_j)\,, 
$$ 
where $\cC_j \cap \tcI$ is considered in $\F_p$ (that is, after reducing elements 
modulo $p$), $j=1, \ldots, h$.
First we note that  
\begin{equation}
\label{eq:cj NIG} 
\sum_{j=1}^h c^2_j \le  \sum_{x\in \cG} r_{\cI^*/\cI^*} (x) = N(\cI^*, \cG) \,, 
\end{equation}
where
$$
N(\cI^*, \cG) = \#\{(x,y)\in \cI^*~:~ x/y \in \cG\}
$$
The quantity $N(\cI^*, \cS)$ has been introduced and studied in~\cite{BKS1} 
however the argument  in~\cite{BKS1} 
 is optimised for the case of rather large subgroups 
of order $T$ around $p^{1/2}$. Thus here we use a bound of Cilleruelo and  Garaev~\cite[Theorem~1]{CillGar} which implies that for $T \le p^{2/5}$ we have
\begin{equation}
\label{eq:NIG bound} 
N(\cI^*, \cS)\le \(X+ TX^2/p + T^{3/4}Xp^{-1/4}\)p^{o(1)}
\end{equation}
 and we have followed the scheme of the proof from 
 paper \cite{BKS1}. 

By~\cite[Equation~(3.13)]{KoSh} we have
$$
t_j \ll T^{2/3} j^{-1/3}, \qquad j=1, \ldots, h\,.
$$  
Hence
\begin{equation}
\label{eq:sum tj} 
 \sum_{j=1}^h t^4_j \ll T^{8/3}.
\end{equation}
Finally, using the Cauchy inequality, we derive from~\eqref{eq:cj NIG}, \eqref{eq:NIG bound} and~\eqref{eq:sum tj} that
$$
R^2 \ll \sum_{j=1}^h t^4_j    \sum_{j=1}^h c^2_j 
	\ll T^{8/3}  \(X+ TX^2/p + T^{3/4}X/p^{1/4}\)p^{o(1)}\,,
$$
 which after substitution in~\eqref{eq:E R}  implies the desired result.
 \end{proof}

\subsection{On the additive energy of polynomial  images}
Given an integer  $k\ge 2$ and sets $\cS_1, \dots, \cS_k \subseteq \F_p$ one can generalize the additive energy 
as  
\begin{align*}
\T^{+}_k (&\cS_1, \ldots, \cS_k)\\
& =\#\{ u_1 + \dots + u_k=v_1+ \dots + v_k~:~  u_i, v_i \in \cS_i, \ i =1, \ldots, k\} \,.
\end{align*}
If $\cS_i = \cS$, $i =1, \ldots, k$ then we write $\T^{+}_k (\cS)$ for $\T^{+}_k (\cS,\dots, \cS)$. 
Clearly, the energy $\T^{+}_k (\cS)$ is translation/dilation invariant and $\T^{+}_2 (\cS) = \E^{+} (\cS)$.  
Using the orthogonality of exponential  functions, we write$$
	\T^{+}_k (\cS_1,\ldots, \cS_k) = \frac{1}{p}  \sum_{\lambda=0}^{p-1}\prod_{j=1}^k \left| \sum_{v_j\in \cS_j} \ep(\lambda v_j)\right|^2
$$
where $\ep(v) = \exp(2\pi i v/p)$  and applying the H{\"o}lder inequality, one sees that 
\begin{equation}\label{f:T_2}
\T^{+}_k (\cS_1, \cS_2, \ldots, \cS_k) \le \prod_{j=2}^{k} \(\T^{+}_k (\cS_1, \cS_j \dots, \cS_j)\)^{1/(k-1)} \,.
\end{equation}

Furthermore, for $k = 2,3, \ldots$ we write
\begin{align*}
\E^{+} (\cS)  &=
\frac{1}{p}  \sum_{\lambda=0}^{p-1}  \left| \sum_{v\in \cS} \ep(\lambda v)\right|^4\\
& \le 
\frac{1}{p}  \sum_{\lambda=0}^{p-1}  \left| \sum_{v\in \cS} \ep(\lambda v)\right|^{2k/(k-1)}
\left| \sum_{v\in \cS} \ep(\lambda v)\right|^{2(k-2)/(k-1)}\,, 
\end{align*}
and using that
$$
\frac{1}{k-1} + \frac{k-2}{k-1} = 1
$$
by the H\"older inequality we derive
\begin{equation}
\label{eq: E2 Tk}
(\E^{+} (\cS))^{k-1} \le  \T^{+}_k (\cS) \T^{+}_1 (\cS)^{k-2} =  \T^{+}_k (\cS) (\# \cS)^{k-2}\,. 
\end{equation}

Now we obtain a nontrivial upper bound for the energy $\T^{+}_k (f(\cA))$ and hence for $\E^{+} (f(\cA)))$
for a non-linear polynomial  $f$ over $\F_p$
in the case when our set $\cA \subseteq \F_p$ has small sum and difference set
$$
\cA \pm \cA = \{a\pm b~:~a,b \in \cA\}\,.
$$
A similar question has been studied, in particular, in~\cite{BT} and~\cite{AMRS}.    
In the proof we follow the method from~\cite{AMRS}. 

\begin{lemma}
\label{lem:pol_images}
Let $f$ be a  polynomial  over $\F_p$ of degree $d\ge 2$. 
Then for  $d=2$ and  sets  $\cA_1, \cA_2, \cA_3 \subseteq \F_p$ with $\# \cA_3 \le \# \cA_1 \le \# \cA_2 \#(\cA_2+\cA_3)$  we have
\begin{align*}
\T^{+}_3   (f (\cA_1) ,   f (\cA_2), f (\cA_3))  
 &\ll 
 \frac{\(\# \cA_1 \# \cA_2 \# (\cA_2+\cA_3)\)^2}{p}\\ 
 &\qquad  \qquad  + \(\# \cA_1 \# \cA_2 \# (\cA_2+\cA_3)\)^{3/2} \,,
\end{align*}
and  for   $d \ge 3$ and a set $\cA \subseteq \F_p$ and we have 
\begin{align*}
\T^{+}_{2^{d-2}+1} (f(\cA)) & \ll \frac{\(\# (\cA-\cA)\)^{2^{d-1}-2} \(\# \cA  \# (\cA+\cA)\)^2}{p}  \\
&\qquad  + \(\# (\cA-\cA)\)^{2^{d-1}-5/2} (\# \cA  \# (\cA+\cA))^{3/2}  \,,
\end{align*}
where the implied constant may depend on $d$.
\end{lemma}

\begin{proof}
For $d=2$, making a linear changing of the variables one can assume that
 $f (Z) = \alpha Z^2 + \beta \in \F_p[Z]$ with $\alpha \neq 0$. 
Then, clearly, the quantity $\T^{+}_3 (f (\cA))$ is equal to the number 
of solutions to 
$$
f(a_1) + 2 \alpha b_1 (b_1+b_2) - \alpha (b_1+b_2)^2 
=
f(a_2) + 2 \alpha c_1 (c_1+c_2) - \alpha (c_1+c_2)^2    \,, 
$$ 
where $a_1,a_2 \in \cA_1$, $b_1,c_1 \in \cA_2$, $b_2, c_2 \in \cA_3$. 
Changing variables $b_1+b_2=u$, $c_1+c_2=v$, leads to the equation 
\begin{align*}
f(a_1) & + 2 \alpha b_1u - \alpha u ^2   =
f(a_2) + 2 \alpha c_1 v - \alpha v^2, \\ &a_1,a_2 \in \cA_1,\ b_1,c_1  \in \cA_2,\ u,v\in \cA_2 +\cA_3    \,.
\end{align*}

We now consider the set of points  
$$
\cQ = \{\(f(a_1) - \alpha u^2 , 2 \alpha u,  c_1\)~:~ a_1 \in \cA_1,\ c_1\in \cA_2,\ u\in \cA_2+\cA_3 \}
$$
and the set of planes
\begin{align*}
\Pi = \{Z_1 + b_1 Z_2 - 2\alpha v Z_3 & = f(a_2) - \alpha v^2  ~:\\
& ~ a_2 \in \cA_1,\ b_1\in \cA_2,\ v\in \cA_2 +\cA_3 \}
\end{align*}
as in  Lemma~\ref{lem:Misha+}. 
Clearly, $\# \cQ  = \# \Pi \ll \# (\cA_2 +\cA_3) \# \cA_1 \# \cA_2$.  
Examining the second and then the first and the third components of points in $\cQ$ 
we see  the maximum number of collinear points in $\cQ$ is  
$$
k \le \max\{ \#\cA_1, \# \(\cA_2 + \cA_3\) \} \,.
$$ 
Using $\# (\cA_2 + \cA_3) \le \# \cA_2 \# \cA_3 \le  \# \cA_1 \# \cA_3$, 
and 
$\# \cA_1 \le \# \cA_2 \#(\cA_2+\cA_3)$
one can check that the second term in the 
  bound of Lemma~\ref{lem:Misha+} never dominates  and we 
obtain the desired result for $d=2$.

Thus, we start with the case $d=3$. 
In fact for the purpose of the induction, we need to establish a more general result. 
In particular, we assume that we are given another set $\cH  \subseteq \F_p$ with 
\begin{equation}
\label{eq:H small} 
\# \cA \ll \# \cH \ll \# (\cA + \cA) \# (\cA - \cA) \,,
\end{equation}
say, 
and we estimate 
 $\T^{+}_3\(\cH, f (\cA), f (\cA)\)$.
 We proceed as in the proof of~\cite[Proposition~2.12,  Claim~(b)]{AMRS}. 
As above  making a change  of the variables one can also assume that  
$f(x) = \alpha x^3 + \beta x + \gamma \in \F_p[Z]$ with $\alpha \neq 0$.

Then, clearly, the quantity  $\T^{+}_3\(\cH, f (\cA), f (\cA)\)$ is equal to the number 
of solutions to 
\begin{align*}
h_1   + 3 &\alpha  (b_1-b_2)\(\frac{(b_1+b_2)^2}{4} + \frac{(b_1-b_2)^2}{12}\) + \beta (b_1-b_2) \\
&  =	h_2+ 3 \alpha  (c_1-c_2)\( \frac{(c_1+c_2)^2}{4} + \frac{(c_1-c_2)^2}{12} \) + \beta (c_1-c_2)\,,\\
&\qquad \qquad h_1,h_2 \in \cH, \ b_1,b_2, c_1,c_2   \in \cA  \,.
\end{align*}
Changing variables $b_1+b_2=u_1$, $b_1-b_2=u_2$, $c_1+c_2=v_1$, $c_1-c_2=v_2$ we 
obtain the equation 
\begin{align*}
h_1  + 3 \alpha  & u_2\(\frac{u_1^2}{4} + \frac{u_2^2}{12}\) + \beta u_2 
=	h_2 + 3 \alpha  v_2\(\frac{v_1^2}{4} + \frac{v_2^2}{12}\) + \beta v_2,\\
&h_1,h_2   \in \cA, \ u_1,v_1\in \cA +\cA , \ u_2,v_2\in \cA -\cA  \,.
\end{align*}
Now using Lemma~\ref{lem:Misha+}, with  the set of points  
\begin{align*}
\cQ & = \{\(h_1 + \beta u_2+   \alpha u_2^2/4, 3 \alpha u_2/4 , v_1^2 \)~:
\\ 
& \qquad  \qquad  \qquad  \qquad  \qquad  \qquad  \quad 
h_1  \in \cH,\ u_2 \in \cA -\cA,\, v_1 \in \cA+\cA  \}
\end{align*}
and the set of planes
\begin{align*}
\Pi &= \{Z_1 + u_1^2 Z_2 - 3\alpha v_2 Z_3/4 = h_2 + \beta v_2 + \alpha v^2_2/4  ~:\\ 
& \qquad  \qquad  \qquad  \qquad  \qquad  \qquad  \quad   h_2 \in \cH,\, v_2 \in \cA -\cA,\, u_1 \in \cA + \cA \}\,.
\end{align*}
Clearly,  $$
\# \cQ = \# \Pi   \ll  \# \cH  \# (\cA -\cA) \# (\cA +\cA)  \,.
$$
We now apply Lemma~\ref{lem:Misha+}, where we have 
$$
k \le  \max\{ \# (\cA -\cA), \# (\cA +\cA), \# \cH\}\,.
$$
Thus 
under the condition~\eqref{eq:H small} 
Lemma~\ref{lem:Misha+} implies that 
\begin{equation}
\label{eq: induct base}
\begin{split}
 \T^{+}_3\(\cH, f (\cA), f (\cA)\)  & \ll \frac{\(\#\cH \# (\cA-\cA) \# (\cA+\cA)\)^2}{p}\\
 &\qquad \quad +\(\#\cH \# (\cA-\cA)\# (\cA+\cA)\)^{3/2} \,, 
\end{split}
\end{equation}
 which in particular gives the desired bound on  $\T^{+}_3\(f (\cA)\)$ for $d=3$.

We know that for any $u\in \F_p$ the following holds  
$$
f(Z+u) - f(Z) = d u g_u (Z) \,,
$$
with some polynomial $g_u \in \F_p [Z]$ depending on $f$ and $u$, of degree $\deg g_u = \deg f - 1 = d-1$. 
Thus $\T^{+}_{2^{d+1}+1} (f(\cA))$ equals the number of solutions to the equation 
\begin{align*}
f(a_1) + d u_1 g_{u_1} &(b_1) + \ldots + d   u_{2^d} g_{u_{2^d}} (b_{2^d}) \\
&= f(a_2) + d v_1 g_{v_1} (c_1) + \ldots + d v_{2^d} g_{v_{2^d}} (c_{2^d}) \,,
\end{align*}
where $a_1,a_2, b_j,c_j \in \cA$ and $u_j,v_j \in \cA - \cA$. 
Hence by~\eqref{f:T_2}, we see that 
\begin{align*}
\T^{+}_{2^{d+1}+1} (f(\cA))& \le \(\# (\cA-\cA)\)^{2^{d+1}} \max_{u\in \cA - \cA} \T^{+}_{2^{d}+1} (f(\cA), g_{u} (\cA), \dots, g_{u}  (\cA))\\
&= \(\# (\cA-\cA)\)^{2^{d+1}} \T^{+}_{2^{d}+1} (f(\cA), f_{d-1} (\cA), \ldots, f_{d-1} (\cA))
\end{align*}
for some polynomial  $f_{d-1}  \in \F_p [Z]$  of degree $d-1$. 
Clearly, using the same argument (which does not make any use of the first set) 
one can obtain in a similar way  that for any set  $\cH  \subseteq \F_p$ we have 
\begin{equation}
\label{eq: induct step}
\begin{split}
\T^{+}_{2^{d}+1} & \(\cH, f_{d-1}  (\cA), \ldots, f_{d-1} (\cA)\) \\
&\quad \le  \(\# (\cA-\cA)\)^{2^{d}} \T^{+}_{2^{d-1}+1} \(\cH, f_{d-2} (\cA), \ldots, f_{d-2} (\cA)\)
\end{split}
\end{equation}
for some polynomial  $f_{d-2} \in \F_p [Z]$  of degree $d-2$. 
Iteratively, using~\eqref{eq: induct step} with $\cH = f(\cA)$,   we derive
\begin{align*}
\T^{+}_{2^{d-2}+1} (f(\cA)) & \le \(\# (\cA-\cA)\)^{2^{d-2} + \ldots + 4} \T_3 (f(\cA), f_{3} (\cA), f_{3} (\cA)) \\
& =  \(\# (\cA-\cA)\)^{2^{d-1} - 4 } \T_3 (f(\cA), f_{3} (\cA), f_{3} (\cA))\,,
\end{align*}
for some cubic  polynomial  $f_{3}  \in \F_p [Z]$. 
Finally, applying bound~\eqref{eq: induct base}  
we obtain
\begin{align*}
\T^{+}_{2^{d-2}+1} (f(\cA)) & \ll \frac{\(\# (\cA-\cA)\)^{2^{d-1}-2} \(\# \cA \# (\cA+\cA)\)^2}{p} \\
&\qquad \quad +
	\(\# (\cA-\cA)\)^{2^{d-1}-5/2}  \(\# \cA \# (\cA+\cA)\)^{3/2} 
\end{align*}
as required. 
\end{proof}

For intervals  $\cA = \cI$, the  statement of Lemma~\ref{lem:pol_images} simplifies as follows.
Clearly,  it is enough to present these bounds only for initial intervals $\cI = [1,X]$.

\begin{cor}
\label{cor:pol_images int}
Let $f$ be a  polynomial  over $\F_p$ of degree $d\ge 2$ and let $\cI = [1,X]$ be an interval
of length $X \le p^{2/3}$.
Then for  $d=2$   we have
$$
\T^{+}_3   \(f (\cI) \) \ll  X^{9/2} \,,
$$
and  for   $d \ge 3$ we have 
$$
\T^{+}_{2^{d-2}+1} \(f(\cI)\) \ll   X^{2^{d-1}+1/2}  \,,
$$
where the implied constant may depend on $d$.
\end{cor}

We now record the bounds on the additive energy of polynomial images which are implied by 
Corollary~\ref{cor:pol_images int} combined with~\eqref{eq: E2 Tk}.

\begin{cor}
\label{cor:pol_images energy}
Let $f$ be a  polynomial  over $\F_p$ of degree $d\ge  2$ and let $\cI = [1,X]$ be an interval 
of length $X \le p^{2/3}$. 
Then  for $d=2$ we have 
$$
\E^{+} \(f (\cI)\)  \ll  X^{11/4}   
$$
and  for   $d \ge 3$ we have 
$$
\E^{+} \(f (\cI)\)  \ll  X^{3-1/2^{d-1}}   
$$
where the implied constant may depend on $d$.
\end{cor}

\begin{rem}
\label{rem:Imrove BT} 
{\rm 
We  recall that
\begin{equation}
\label{eq:Triangle}
\#(\cA-\cA ) \le \frac{\(\#\(\cA+\cA\)\)^2}{\#\cA}, 
\end{equation}
which follows from the Ruzsa triangle inequality~\cite[Chapter~1, Theorem~8.1]{Ruzsa}, see also~\cite[Lemma~9]{BT}).
From Lemma~\ref{lem:pol_images} together with~\eqref{eq:Triangle}
one can derive that if $\# \cA \#(\cA+ \cA) \# (\cA - \cA) \le p^2$, then 
for $d\ge 3$ 
$$
	\#(\cA+\cA )^{4-14/2^{d}} \#(f(\cA)+ f(\cA) ) \gg |A|^{5-6/2^{d}}  
$$
and thus
$$
	\#(\cA+\cA ) + \#(f(\cA)+ f(\cA) ) \gg |A|^{1+1/(5\cdot 2^{d-1}-7)} .
$$
This improves~\cite[Theorem~1]{BT} which gives the exponent $1 + 1/(16\cdot 6^d)$
and under a more stringent condition $\# \cA \le p^{1/2}$.  
We also have a similar result for $d=2$. 
}
\end{rem}

\section{Proofs of main results}
 
\subsection{Proof of Theorem~\ref{thm:Gen Set E3}}

We have
$$
\left| W_\chi(\cI, \cS;\balpha, \bbeta)\right|  \le \sum_{x \in \cI}  \left| \sum_{s\in \cS}   \alpha_s     \chi(s+x) \right| .
$$
Thus, using $\overline \chi$ to denote the complex conjugate character to $\chi$, by the Cauchy inequality we derive
\begin{equation}
\label{eq:WV}
\begin{split}
\left| W_\chi(\cI, \cS;\balpha, \bbeta)\right|^2 & \le X \sum_{x \in \cI}  \left| \sum_{s\in \cS}   \alpha_s     \chi(s+x) \right|^2\\
& = X \sum_{x \in \cI}   \sum_{s,t \in \cS }   \alpha_s  \overline{\alpha_t}  \chi(s+x) \overline \chi(t+x) \\
& = XV + O(SX^2)\,,
  \end{split}
\end{equation}
where  
$$
V  =   \sum_{\substack{s,t \in \cS \\s\ne t}}   \alpha_s  \overline{\alpha_t}  \sum_{x \in \cI}   \chi(s+x) \overline \chi(t+x).  
$$

We fix some integers $Y, Z\ge 1$ with $4YZ \le X$ and denote by $\cY$ the set of primes of the interval $[Y,2Y]$.

Applying the same transformation as in the work of Fouvry 
and Michel~\cite[Equations~(4.3) and~(4.4)]{FouMich} and write
$$
V \le \frac{p^{o(1)}}{YZ}\sum_{\substack{s,t \in \cS \\s\ne t}}  \sum_{y \in \cY} \sum_{x \in \bcI} 
\left| \sum_{z=Z+1}^{2Z} \eta_z  \chi(s+x+yz) \overline \chi(t+x+yz)  \right| 
$$
with some complex numbers $\eta_z$ satisfying $|\eta_z| = 1$ and the new interval 
$$
\bcI = [-X,X]\, .
$$

Now, using the multiplicativity of $\chi$,   we obtain 
\begin{equation}
\label{eq:V-1}
V \le \frac{p^{o(1)}}{YZ}\sum_{\substack{s,t \in \cS \\s\ne t}}  \sum_{x \in \bcI}   \sum_{y\in \cY}
\left| \sum_{z=Z+1}^{2Z} \eta_z   \chi\(\frac{s+x}{y}+z\) \overline \chi\( \frac{t+x}{y}+ z\) \right| .
\end{equation}
For each pair $(\lambda, \mu) \in \F_p^2$ we denote by $\nu(\lambda, \mu)$ the number 
of solutions to the system of equation
$$
\frac{s+x}{y} = \lambda, \quad \frac{t+x}{y} = \mu, \qquad
( s,t,x,y) \in \cS^2 \times\bcI\times \cY,  \ s\ne t.
 $$
Thus we can re-write~\eqref{eq:V-1} as 
\begin{equation}
\label{eq:V-2}
V \le \frac{p^{o(1)}}{YZ}\sum_{(\lambda, \mu) \in \F_p^2} \nu(\lambda, \mu)
\left| \sum_{z=Z+1}^{2Z} \eta_z   \chi\(\lambda+z\) \overline \chi\( \mu+ z\) \right| \,.
\end{equation}

Clearly, 
$$
\sum_{(\lambda, \mu) \in \F_p^2} \nu(\lambda, \mu) \ll S^2XY
\mand 
\sum_{(\lambda, \mu) \in \F_p^2} \nu(\lambda, \mu)^2= N(\cS, \bcI, \cY). 
 $$
 We now fix some integer $r\ge 1$ and write 
$$
\nu(\lambda, \mu)  =  \nu(\lambda, \mu)^{1-1/r}  \(\nu(\lambda, \mu)^{2}\)^{1/2r}.
$$
Applying the H{\"o}lder inequality we derive from~\eqref{eq:V-2} that 
\begin{align*}
V^{2r}  &\le \frac{p^{o(1)}}{Y^{2r}Z^{2r}}
\(\sum_{(\lambda, \mu) \in \F_p^2} \nu(\lambda, \mu) \)^{2r-2}
\sum_{(\lambda, \mu) \in \F_p^2} \nu(\lambda, \mu)^2\\
& \qquad \qquad \qquad \qquad \quad \sum_{(\lambda, \mu) \in \F_p^2}
\left| \sum_{z=Z+1}^{2Z} \eta_z   \chi\(\lambda+z\) \overline \chi\( \mu+ z\) \right|^{2r}\\
 &\le \frac{S^{4r-4} X^{2r-2} p^{o(1)}}{Y^{2}Z^{2r}} N(\cS, \bcI, \cY)\\
& \qquad \qquad \qquad \qquad \quad \sum_{(\lambda, \mu) \in \F_p^2}
\left| \sum_{z=Z+1}^{2Z} \eta_z   \chi\(\lambda+z\) \overline \chi\( \mu+ z\) \right|^{2r}. 
\end{align*}
By the  condition~\eqref{eq:cond SX}
we see that Lemma~\ref{lem:Bound NSXY} applies and yields
\begin{equation}\label{eq:V-3}
V^{2r}\le   \frac{S^{4r-4} X^{2r-2}}{Y^{2}Z^{2r}}  \( Y \E^{+}_3 (\cS, \cS, \bcI)+S^3 X^{3/2}+S^2X^2\) \sigma   p^{o(1)} , 
\end{equation}
where 
$$
\sigma = \sum_{(\lambda, \mu) \in \F_p^2}
\left| \sum_{z=Z+1}^{2Z} \eta_z   \chi\(\lambda+z\) \overline \chi\( \mu+ z\) \right|^{2r}.
$$
Furthermore, expanding and changing the order of summation, we derive
\begin{align*}
\sigma
&=  \sum_{(\lambda, \mu) \in \F_p^2}
\sum_{z_1, \ldots z_{2r}=Z+1}^{2Z} 
\prod_{i=1}^r \eta_{z_i} \chi\(\lambda+z_i\) \overline \chi\( \mu+ z_i\)  \\
& \qquad \qquad  \qquad \qquad  \qquad \qquad 
 \prod_{i=r+1}^{2r} \overline \eta_{z_i} \overline \chi\(\lambda+z_i\)  \chi\( \mu+ z_i\) \\
 & \le  \sum_{z_1, \ldots z_{2r}=Z+1}^{2Z} \left| \sum_{\lambda \in \F_p} 
 \prod_{i=1}^r   \chi\(\lambda+z_i\)  \overline \chi\(\lambda+z_{r+i}\) \right|^2.
 \end{align*}

Using the Weil bound in the form given by~\cite[Corollary~11.24]{IwKow} if
$(z_1, \ldots, z_r)$ is not a permutation of $(z_{r +1}, \ldots, z_{2r})$, 
and the trivial bound otherwise, we derive
\begin{equation}
\label{eq:sum 2r}
\sigma
 \ll Z^{2r} p +  Z^{r} p^2, 
\end{equation}
(see also~\cite[Lemma~12.8]{IwKow} that underlies the Burgess method). 

We now choose 
$$Y = \fl{2X p^{-1/r}} \mand Z = \fl{p^{1/r}},
$$ 
(note that due to the condition $X \ge p^{1/r}$ this is an admissible choice), 
so that~\eqref{eq:sum 2r} becomes
$$
\sigma \ll Z^{2r} p,
$$
which after the  substitution in~\eqref{eq:V-3} becomes 
$$
V^{2r}\le   \frac{S^{4r-4} X^{2r-2}}{Y^{2}} 
\( Y \E^{+}_3 (\cS, \cS, \bcI)+S^3 X^{3/2}+S^2 X^2 \)
p^{1+o(1)}\,.
$$
In turn, substituting this in~\eqref{eq:WV} after simple calculations  we derive 
 the desired result.

\subsection{Proof of Theorem~\ref{thm:Subg}}

Clearly,  we need the conditions
\begin{equation}
\label{eq:cond1}
5\zeta + 2\xi > 2 \mand  \zeta+\xi > 1/2 
\end{equation}
to make sure that the terms in the bound of Theorem~\ref{thm:Gen Set E3}
that do not depend on $\E^{+}_3 (\cS, \cS, \bcI)$ are nontrivial (provided that 
$r$ is large enough depending only on $\zeta$ and $\xi$).

Now substituting the first bound  of Lemma~\ref{lem:E3 subgr}
into  Theorem~\ref{thm:Gen Set E3} obtain a nontrivial result provided
\begin{equation}
\label{eq:cond2}
40\zeta  + 31\xi > 20
\end{equation}
and  $r$ is large enough.

 It is  also useful to observe that since $TX \le p^{2/5 + 1/2+o(1)}$ the second 
 bound of Lemma~\ref{lem:E3 subgr} simplifies as 
 $$
  \E^{+}_3 (\cG,  \cG, \cI) \le 
 \( T^2 X +  T^{4/3} X^{3/2}  + T^{41/24} X^{3/2}p^{-1/8}\)  p^{o(1)}\, .
 $$
Hence, substituting this bound   
into  Theorem~\ref{thm:Gen Set E3} obtain a nontrivial result provided
\begin{equation}
\label{eq:cond3}
9\zeta  + 16 \xi > 6 \mand  
36 \zeta+ 55\xi > 21
\end{equation}
(again for a sufficiently large   $r$).

So,  to have a nontrivial estimate, the parameters $\zeta$ and $\xi$ must satisfy~\eqref{eq:cond1} and at least one out of~\eqref{eq:cond2} and~\eqref{eq:cond3}.  Now, after simple, but somewhat tedious, calculations one derives the desired result. 

\subsection{Proof of Theorem~\ref{thm:Primes}}

For the first sum we write 
$$
 \sum_{\substack{q\le Q\\q~\text{prime}}}  \left|  \sum_{\substack{r \le R\\ r~\text{prime}}}
 \chi(f(q) +r)\right| = 
 \sum_{\substack{q\le Q\\q~\text{prime}}} e^{i \psi_q} \sum_{\substack{r \le R\\ r~\text{prime}}}
 \chi(f(q) +r)
$$
where $0 \le \psi_q < 2\pi$ is the argument of the second sum (which depends only on $q$). 
We introduce the weights $\alpha_s$ and $\beta_x$, where 
\begin{itemize}
\item 
$\alpha_s$ is supported 
on the set $\cS = \{f(q) ~:~q\le Q,\ q~\text{prime}\}$ and is defined as
$$
\alpha_s = \frac{1}{d} \sum_{\substack{q\le Q,\, f(q) = s\\q~\text{prime}}} e^{i \psi_q}\,,
$$
thus $|\alpha_s| \le 1$;
 \item  $\beta_x$ is the characteristic  function of primes $r \le R$.
\end{itemize}
With the above notations, 
$$
 \sum_{\substack{q\le Q\\q~\text{prime}}}  \left|  \sum_{\substack{r \le R\\ r~\text{prime}}}
 \chi(f(q) +r)\right| = d W_\chi(\cI, \cS;\balpha, \bbeta).
 $$
 We also use a similar representation for the second sum.

The result is instant if one combines Corollary~\ref{cor:Gen Set Energy} 
with the bound  on the additive energy of polynomial images from 
Corollary~\ref{cor:pol_images energy}.

 \section*{Acknowledgement}

The authors are grateful to Changhao Chen and  Bryce Kerr for their comments on the 
initial draft.

During the preparation of this work the first  author was supported in part by 
the Program of the Presidium of the Russian Academy of Sciences~–01 ``Fundamental Mathematics and its Applications' under grant PRAS-18-01'' 
and the  second author   by 
the  Australian Research Council  Grants DP170100786 and DP180100201.

\end{document}